\documentclass[12pt]{article}
\usepackage{graphicx} % Required for inserting images
\usepackage{tikz-cd}
\usepackage[utf8]{inputenc}
\usepackage{amsmath}
\usepackage{amsfonts}
\usepackage{amssymb}
\usepackage{amsthm}
\usepackage{booktabs}
\usepackage{latexsym}
\usepackage{ytableau}
\usepackage{tikz}
\usepackage[
    backend=biber,
    style=numeric,
  ]{biblatex}
\usepackage{float}
\usepackage{mathdots}
\usepackage[multiple]{footmisc}

\newtheorem{thm}{Theorem}

\newtheorem{defin}{Definition}
\newtheorem{rmk}{Remark}
\newtheorem{conj}{Conjecture}
\newtheorem{cor}{Corollary}
\newtheorem{lem}{Lemma}

\newcommand{\li}{\emph{Li}}

\title{A simple inequality relating the Euler-Riemann zeta function, digamma, and cotangent over the unit interval}
\author{Michael Andrew Henry [TU Graz]}
\date{June 23, 2025}

\begin{document}

\maketitle

\begin{abstract}
    We prove an inequality featuring three well-known functions from analysis, namely the cotangent, the Euler-Riemann zeta function, and the digamma function. Aside from a simple proof of our result, we give a conjectured strengthening. We offer various remarks about the origins of this problem.
\end{abstract}

%%#&*()_+))(()#&(*&*#__
\section{Introduction}
%%#&*())#*&^%^&*()(*#((

The Riemann $ \zeta $-function is the special function
    \begin{eqnarray*}
        \zeta(s) := \sum^{\infty}_{n=1}\frac{1}{n^s}
    \end{eqnarray*}
where $  \Re(s) > 1 \in \mathbb{C}  $. Naturally, the most famous problem concerning $ \zeta $ is the Riemann hypothesis; however, recently authors have been studying various problems involving the $ \zeta(s) $ when $ s $ is restricted so some subset of $ \mathbb{R} $. As a sample of such articles, see e.g. \cite{alkan2019}, \cite{alzerKwong2025}, \cite{alzerKwong2021}, \cite{alzer2005}, \cite{cerone2007}, \cite{ceroneEtAll2004}, \cite{delange1987},  \cite{hilberdink2023}. We produce our own variety of such a result by giving a new inequality involving $ \zeta $-function restricted to the unit interval. 

Throughout this note, the digamma function is $ \psi(z) = \Gamma'(z)/\Gamma(z) $, where $ \Gamma(z) $ is Euler's generalization of the factorial function. What we prove is the following
    \begin{thm}\label{thm:cotVsZeta&Digamma}
        For $x$ satisfying $ 0 < x < 1 $, the relation 
            \begin{eqnarray*}
                \pi \cot \pi x < \zeta(x) - \psi(x)
            \end{eqnarray*}
        holds. 
    \end{thm}

%%#&*()__)__)#_)_
\section{Remarks}
%%#^&&#%$%#^&()($

Briefly, we give some remarks about the problem and method of solution. The appearance in Theorem \ref{thm:cotVsZeta&Digamma} of the trigonometric function $ \cot $ is not entirely surprising since, e.g., using Riemann's functional equation of $ \zeta(s) $, namely
    \begin{eqnarray}
       \zeta(1-s) = \frac{2\,\Gamma(s)\cos(\pi s/2)}{(2\pi)^{s}} \zeta(s) ,
    \end{eqnarray}
or equivalently 
    \begin{eqnarray*}
        \zeta(s) = \frac{2\,\Gamma(1-s)\sin(\pi s/2)}{(2\pi)^{1-s}}\zeta(1-s),
    \end{eqnarray*}
(as found in \cite{apostol1976I}, p. 259), we can write, formally,
    \begin{eqnarray}\label{eqn:cotInTermsOfZetaGamma}
        {\pi} \cot {\pi s}
        = \left(\frac{(2\pi)^{2s}\,\zeta(1-2s)}{\zeta(2s)}\right)^2\frac{\Gamma(1-2s)}{2\,\Gamma(2s)}.
    \end{eqnarray}
What is more interesting is Theorem \ref{thm:cotVsZeta&Digamma} relates a simple additive function of $ \zeta(x) $ and $\psi(x) $ with a ratio of $\zeta(x)$ and $ \Gamma(x) $ as in \eqref{eqn:cotInTermsOfZetaGamma} (with the domain restricted appropriately). 

We have written Theorem \ref{thm:cotVsZeta&Digamma} in terms of $ \cot $ because this function satisfies the -- rather strong -- replicative property (probably first defined as this in \cite{knuth1968}, see also \cite{knuth1981}). Sometimes a function having this so-called replicative property is called a Kubert function (due to \cite{kubert1979}). We alter the naming slightly in a definition.  
    \begin{defin}\label{defin:kubertKnuth}
        We call a function $ f(x) $ a \emph{Kubert-Knuth} function of weight $ \omega $ over $ (0,1) $ if it satisfies the infinite functional equation 
            \begin{eqnarray}\label{eq:kubertKnuthFunctionalEq}
                \frac{1}{p^{\omega}}\sum_{k=0}^{p-1}f\left(\frac{x + k}{p}\right) = f(x)
            \end{eqnarray}
        for all $ p \in 
        \mathbb{N} $ and some fixed $ \omega \in \mathbb{C} $ dependent on $ f $.
    \end{defin}

Kubert-Knuth functions have the following series interpretation that helps clarify the significance of the functional equation. Suppose that $ f(x) = \sum_{k=0}^{\infty} a_nq^n $ is the Fourier expansion of $ f $ (i.e. $ q := e^{2 \pi i x} $) on the interval $ (0,1) $. Then, if $ f $ is a Kubert-Knuth function, it follows that, for arbitrary $p \in \mathbb{N} $, relation
    \begin{eqnarray}\label{eq:KubertSeries}
        \sum_{n=1}^{\infty} a_n q^n  = \sum_{n=0}^{\infty}a_{pn}q^{n}    
    \end{eqnarray}
holds. We can illustrate this with an example. The function $ \cot \pi x $ up to complex multiples and an additive constant is perhaps the archetypal Kubert-Knuth function, and also its Fourier expansion is well-known. Using the geometric series
    \begin{eqnarray}
        \frac{1}{1-x} := \sum_{n=0}^{\infty}x^n
    \end{eqnarray}
valid for $ |x|<1 $, then, clearly, if $ a_n = 1 $, identity
    \begin{eqnarray*}
        \frac{1}{1-x} = \sum_{n=0}^{\infty}a_n x^n = \sum_{n=0}^{\infty}a_{pn}x^n
    \end{eqnarray*}
holds for an arbitrary $ p \in \mathbb{N} $. Substitution $ x \mapsto q = e^{2\pi i x } $ renders
    \begin{eqnarray*}
        \frac{1}{1-q} = \frac{1}{2}\left(1 + i \cot \pi x\right)
    \end{eqnarray*}
true. After a normalization, we get the known Fourier expansion 
    \begin{eqnarray*}
        \frac{1 + q}{1-q} = i\cot \pi x = 1+2\sum_{n=1}^{\infty}q^{n}.
    \end{eqnarray*}
This also suggests the close relationship between Kubert-Knuth functions and polylogarithms of order $s$ or those functions defined (for our purposes)
    \begin{equation*}
        \emph{\li}_s(z) := \sum_{n=1}^{\infty} \frac{z^n}{n^s},
    \end{equation*}
where $ 1 \geq z \in \mathbb{C} $ and $s \in \mathbb{C}$ is arbitrary. Observe that $ 1/(1-x) = \emph{\li}_0(x)/x $ holds, etc. See also the related studies \cite{LagariasLi2016i} and \cite{LagariasLi2016ii}. 
    \begin{rmk}
         The functions we have defined as Kubert-Knuth functions are typically called replicative as in \cite{kairies1997}, \cite{kairies2000}, \cite{knuth1968}, and \cite{knuth1981}, or {Kubert} functions as in \cite{kanemitsuYoshimoto1996} and \cite{milnor1983}; however, sometimes the expression of the Kubert-Knuth property at \eqref{eq:kubertKnuthFunctionalEq} is solely called a multiplication formula e.g. see \cite{srivChoi2012}, p. 60. Differently still, discussion of this mutliplication formula as a functional equation is found in \cite{kuczmaEtAl1990} (also see \cite{hilberdink2001}). Thus, papers about Kubert-Knuth functions are scattered under various terminologies, and articles often overlap or give incomplete bibliographies. These functions are very interesting, appear in many guises, and deserve an encyclopedic analysis yet to be written.   
    \end{rmk}

Above, we claimed the Kubert-Knuth property  is strong. This is due to the fact that it characterizes $ \cot \pi x $ on the interval $ (0,1) $ when $ \omega = 1 $ in Definition \ref{defin:kubertKnuth} (see e.g. \cite{jager1985} and \cite{milnor1983}). This is true also of other Kubert-Knuth functions such as the restricted Bernoulli polynomials (see \cite{carlitz1953}), etc. This completes our discussion of Kubert-Knuth functions.

%+#@_)(*)))#**#
\section{Proof}
%%#&()_)(*&#()_

We next give some remarks to reduce what must be shown. Then well-known identity $$ \psi(1-x)-\psi(x) = \pi \cot \pi x $$ is classical (or can be found in \cite{srivChoi2012}, p. 14), and so the inequality of Theorem \ref{thm:cotVsZeta&Digamma} is equivalent to showing, for $ 0 < x < 1 $, relation
    \begin{eqnarray}\label{eqn:CotReduction}
        \psi(1-x) < \zeta(x)
    \end{eqnarray}
holds. Our method of proving \eqref{eqn:CotReduction}, then, is to find some $ f $, where a ``squeezing'' argument amounts to showing
    \begin{eqnarray}\label{eqn:squeeze}
        \psi(1-x) + \frac{1}{1-x}< f(x) \leq \zeta(x)+\frac{1}{1-x}
    \end{eqnarray}
holds. This path was suggested by the following beautiful 
    \begin{thm}[Elezovi\'{c}, Giordano, and Pe\v{c}ari\'{c}]\label{thm:digammaBound}
        Let $\gamma = - \psi(1) $. Whenever $ 0 < x < \infty $, the inequality
            \begin{equation*}
                \log\left(x+\frac{1}{2}\right)-\frac{1}{x} \leq \psi(x) \leq \log\left(x+\frac{1}{e^{\gamma}}\right)-\frac{1}{x}.
            \end{equation*}
        holds.
    \end{thm}
    \begin{proof}
        See \cite{ElezovichEtAl200}.
    \end{proof}
From Theorem \ref{thm:digammaBound}, the following comes out effortlessly as a special case:
    \begin{cor}\label{cor:auxF}
        If $ 0 < x <  1$ holds, then
            \begin{equation*}
                \psi(1-x)  + \frac{1}{1-x} \leq \log \left(1-x + \frac{1}{e^\gamma}\right)
            \end{equation*}
        holds. 
    \end{cor}

What we take from the Corollary is the following
    \begin{lem}\label{lem:decreasing}
        On the interval $ 0 <x < 1$, the function
            \begin{eqnarray}\label{eqn:lineBelowZeta}
                \log \left(1-x + \frac{1}{e^\gamma} \right)
            \end{eqnarray}
        is monotone decreasing. 
    \end{lem}
    \begin{proof}
        We do this by showing the derivative is negative over the unit interval. We find
            \begin{eqnarray*}
                \frac{d}{dx}\log(1-x + e^{-\gamma}) = - \frac{1}{1-x + e^{-\gamma}}
            \end{eqnarray*}
        holds and the conclusion is immediate by observation.
    \end{proof}
Having established Lemma \ref{lem:decreasing}, then the function $ f(x) $ from \eqref{eqn:squeeze} we take as
    \begin{eqnarray*}
        f(x) := b x + \frac{1}{2}
    \end{eqnarray*}
where $ b = B_1(\gamma) = \gamma-\frac{1}{2} $, $B_1(x) $ the first Bernoulli polynomial. This choice of $ f(x) $ is evidently monotone increasing on the unit interval, contrary to $ \log(1-x + e^{-\gamma}) $ by Lemma \ref{lem:decreasing}. Thus, considering that 
    \begin{eqnarray*}
        \lim _{x\rightarrow 0} \log \left(1-x + \frac{1}{e^\gamma}  \right)= \log(1+e^{-\gamma}) = 0.44... 
    \end{eqnarray*}
holds and 
    \begin{eqnarray}\label{eqn:linearSqueezeFunction}
        \lim_{x \rightarrow \alpha}  f(x) = 
            \begin{cases}
                \frac{1}{2} \quad \text{if } \alpha = 0 \\
                \gamma \quad \text{if } \alpha = 1
            \end{cases}
    \end{eqnarray}
holds, the relation $ \log(1-x + e^{-\gamma})\leq f(x) $ holds, as needed.

By the discussion above, we have made an easy reduction of the proof of Theorem \ref{thm:cotVsZeta&Digamma} to showing, if $ 0 < x < 1 $, then the relation 
    \begin{eqnarray*}
        f(x) \leq \zeta(x) + \frac{1}{1-x}
    \end{eqnarray*}
is satisfied. So, this is the contents of the next  
    \begin{proof}
        Firstly, we show that on interval $ (0,1) $ the function $ \zeta(x) + 1/(1-x) $ is strictly increasing. For this, recall the following classical formula: when $ \Re(s)> -1 $ holds, the relation
            \begin{eqnarray*}
                \zeta(s) + \frac{1}{1-s} = \frac{1}{2} - s \int^{\infty}_{1}\frac{\{t\}-\frac{1}{2}}{t^{s+1}} dt
            \end{eqnarray*}
        is true (see \cite{srivChoi2012}, p. 144). For  $ x $ satisfying $ 0 < x < 1 $, we define 
            \begin{eqnarray*}
                I(x) := \int^{\infty}_{1}\frac{\{t\}-\frac{1}{2}}{t^{x+1}} dt.
            \end{eqnarray*}
        Alternatively, if
            \begin{eqnarray*}
                I_n(x) := \int_n^{n+1}\frac{t-n - \frac{1}{2}}{t^{x+1}}dt,
            \end{eqnarray*}
        then
            \begin{eqnarray*}
                I(x) := \sum_{n=1}^\infty I_n(x).
            \end{eqnarray*}
        From the mere definition of $ I(x)$, it follows by observation that if $ 0 < x_0 < x_1 < 1 $, then $ - \infty < I(x_1) < I(x_0) < 0 \Longleftrightarrow \infty >  -I(x_1) > - I(x_0) > 0 $ holds. Therefore, we find
            \begin{eqnarray*}
                \zeta(x) + \frac{1}{1-x} = x|I(x)| + \frac{1}{2} 
            \end{eqnarray*}
        is as we wish i.e. strictly increasing. At the endpoints, consider the classical (and easily established) fact that
            \begin{eqnarray*}
                \lim_{x \rightarrow \alpha} \left( \zeta(x) + \frac{1}{1-x}\right) = 
                    \begin{cases}
                        \frac{1}{2} \quad \text{if } \alpha = 0 \\
                        \gamma \quad \text{if } \alpha = 1
                    \end{cases}
            \end{eqnarray*}
        holds (see e.g. \cite{srivChoi2012}, pp. 91-92). Thus, we can conclude from \eqref{eqn:linearSqueezeFunction} that
            \begin{eqnarray*}
                f(x) = b x + \frac{1}{2} \leq \zeta(x) + \frac{1}{1-x} 
            \end{eqnarray*}
        is true. This completes the proof.
    \end{proof}
    \begin{rmk}
        The fact that for $ x \in \{0<x<1: x \in \mathbb{R}\} $,
            \begin{eqnarray*}
                \zeta(x) + \frac{1}{1-x}
            \end{eqnarray*}
        is strictly increasing was proven in \cite{hilberdink2023} using a series truncation. 
    \end{rmk}

%%($)_*$()_))0-49947
\section{Conclusion}
%#%$$%^@&$$$%@#^&*($

In the course of studying the above relation, we found numerous additional problems. The one that is most closely related to our Theorem \ref{thm:cotVsZeta&Digamma} we give to the reader as
    \begin{conj}
         In this conjecture we define $ \overline{b} := \overline{B_1(\gamma)} = \gamma + \frac{1}{2}$ and $ b = B_1(\gamma) := \gamma - \frac{1}{2} $ (again, $ B_1(x) $ is the first Bernoulli polynomial and $ \gamma = -\psi(1) $ is the Euler-Mascheroni constant). Then if $ 0 < x < 1 $, we conjecture that the relation 
            \begin{equation*}
                \pi \cot \pi x + x < \zeta(x) - \psi(x) < \pi \cot \pi x + \overline{b}x + b
            \end{equation*}  
        holds.
    \end{conj}
Using the series interpretation of a Kubert-Knuth function, if the conjecture is true, it suggests bounds on the coefficients of the Fourier expansion of  
    \begin{eqnarray*}
        \zeta(x) - \psi(x),
    \end{eqnarray*}
periodicity being achieved by the simple trick of considering only the fractional part of a real number. This conjecture can be related to an elementary criterion of the Riemann hypothesis. The novelty would be that, assuming the conjecture is true, the Riemann hypothesis could be ``reduced'' to a question about the behavior of $ \zeta(x) $ over $ 0 < x < 1 $. Because this is quite technical to state and we do not expect this to be any easier, we leave the details to the reader (related analyses in full detail are given in \cite{kanemitsuYoshimoto1996}, \cite{yoshimoto2004}).

\printbibliography 

\end{document}